\documentclass[12pt, reqno]{jams-l}

\usepackage{amsmath,enumerate,amsfonts,amssymb,color,graphicx,amsthm}

\usepackage[colorlinks=true, allcolors=black]{hyperref}
\usepackage{todonotes}
\usepackage[normalem]{ulem}
\numberwithin{equation}{section}
\usepackage{cite}
\usepackage{mathtools}
\usepackage[toc,page]{appendix}
\usepackage[margin=3cm]{geometry}

\newlength{\leftstackrelawd}
\newlength{\leftstackrelbwd}
\def\leftstackrel#1#2{\settowidth{\leftstackrelawd}%
	{${{}^{#1}}$}\settowidth{\leftstackrelbwd}{$#2$}%
	\addtolength{\leftstackrelawd}{-\leftstackrelbwd}%
	\leavevmode\ifthenelse{\lengthtest{\leftstackrelawd>0pt}}%
	{\kern-.5\leftstackrelawd}{}\mathrel{\mathop{#2}\limits^{#1}}}

\theoremstyle{plain}
\newtheorem{thm}{Theorem}[section]
\newtheorem{lem}[thm]{Lemma}
\newtheorem{cor}[thm]{Corollary}

\newtheorem*{thm*}{Theorem}

\theoremstyle{definition}

\newtheorem{rmk}[thm]{Remark}
\newtheorem{?}[thm]{Problem}
\newtheorem{Ex}[thm]{Example}
\newtheorem{innercustomthm}{Theorem}
\newenvironment{customthm}[1]
{\renewcommand\theinnercustomthm{#1}\innercustomthm}
{\endinnercustomthm}

\renewcommand{\phi}{\varphi}
\renewcommand{\epsilon}{\varepsilon}
\makeatletter
\def\@cite#1#2{[\textbf{#1\if@tempswa , #2\fi}]}
\def\@biblabel#1{[\textbf{#1}]}
\makeatother

\makeatletter
\newcommand*{\defeq}{\mathrel{\rlap{%
			\raisebox{0.3ex}{$\m@th\cdot$}}%
		\raisebox{-0.3ex}{$\m@th\cdot$}}%
	=}
\makeatother

\makeatletter
\newcommand*{\eqdef}{=\mathrel{\rlap{%
			\raisebox{0.3ex}{$\m@th\cdot$}}%
		\raisebox{-0.3ex}{$\m@th\cdot$}}%
}
\makeatother

\newcounter{marnote}

\makeatletter
\def\underbracex#1#2{\mathop{\vtop{\m@th\ialign{##\crcr
				$\hfil\displaystyle{#2}\hfil$\crcr
				\noalign{\kern3\p@\nointerlineskip}%
				#1\crcr\noalign{\kern3\p@}}}}\limits}

\def\upbracefilla{$\m@th \setbox\z@\hbox{$\braceld$}%
	\bracelu\leaders\vrule \@height\ht\z@ \@depth\z@\hfill 
	\kern\p@\vrule \@width\p@\kern\p@\vrule \@width\p@\kern\p@\vrule \@width\p@
	$}

\def\upbracefillb{$\m@th \setbox\z@\hbox{$\braceld$}%
	\vrule \@width\p@\kern\p@\vrule \@width\p@\kern\p@\vrule \@width\p@\kern\p@
	\leaders\vrule \@height\ht\z@ \@depth\z@\hfill\bracerd
	\braceld\leaders\vrule \@height\ht\z@ \@depth\z@\hfill
	\kern\p@\vrule \@width\p@\kern\p@\vrule \@width\p@\kern\p@\vrule \@width\p@
	$}

\def\upbracefillc{$\m@th \setbox\z@\hbox{$\braceld$}%
	\vrule \@width\p@\kern\p@\vrule \@width\p@\kern\p@\vrule \@width\p@\kern\p@
	\leaders\vrule \@height\ht\z@ \@depth\z@\hfill
	\kern\p@\vrule \@width\p@\kern\p@\vrule \@width\p@\kern\p@\vrule \@width\p@
	$}

\def\upbracefilld{$\m@th \setbox\z@\hbox{$\braceld$}%
	\vrule \@width\p@\kern\p@\vrule \@width\p@\kern\p@\vrule \@width\p@\kern\p@
	\leaders\vrule \@height\ht\z@ \@depth\z@\hfill\braceru$}

\def\upbracefillbd{$\m@th \setbox\z@\hbox{$\braceld$}%
	\vrule \@width\p@\kern\p@\vrule \@width\p@\kern\p@\vrule \@width\p@\kern\p@
	\bracerd\braceld
	\leaders\vrule \@height\ht\z@ \@depth\z@\hfill\braceru$}

\makeatother

\makeatletter
\def\l@subsection{\@tocline{2}{0pt}{2.5pc}{5pc}{}}
\makeatother

\begin{document}

	\title[Interior regularity for strong solutions]{Interior regularity for strong solutions to a class of fully nonlinear elliptic equations} 
	
	\author{Jonah A. J. Duncan}
	\address{Johns Hopkins University, 404 Krieger Hall, Department of Mathematics, 3400 N. Charles Street, Baltimore, MD 21218, US.}
	\curraddr{}
	\email{jdunca33@jhu.edu}
	\thanks{}

	\date{}

	\maketitle

	\vspace*{-5mm}\begin{abstract}
		We obtain local pointwise second derivative estimates for $W^{2,p}$-strong \linebreak solutions to a class of fully nonlinear elliptic equations on Euclidean domains,\linebreak motivated by problems in conformal geometry.

		\vspace*{1mm}
		
		%	\medskip
		%	Keywords: xxxxx
		
		%	\medskip
		%	MSC: xxxxx
	\end{abstract}

\section{Introduction}

Let $\Omega\subset\mathbb{R}^n$ ($n\geq 2$) be a domain. For a positive function $u:\Omega\rightarrow \mathbb{R}$, denote by $A_u$ the symmetric matrix-valued function
\begin{equation*}
A_u = \nabla^2 u - \frac{|\nabla u|^2}{2u}I,
\end{equation*}
where $I$ is the $n\times n$ identity matrix, and denote by $\lambda(A_u(x))\in\mathbb{R}^n$ the eigenvalues of the matrix $A_u(x)$. In this paper, we obtain local pointwise second derivative estimates for positive solutions $u\in W^{2,p}_{\operatorname{loc}}(\Omega)\cap C^{0,1}_{\operatorname{loc}}(\Omega)$ to equations of the form
\begin{equation}\label{1}
f(\lambda(A_u(x))) = \psi(x,u(x))>0, \quad \lambda(A_u(x))\in\Gamma \quad\text{for a.e. }x\in\Omega,
\end{equation}
where $f$ and $\Gamma$ are assumed to satisfy the following standard properties\footnote{We note that, given $\Gamma$ satisfying \eqref{21} and \eqref{22}, there exists a defining function $f$ satisfying \eqref{23} and \eqref{24} -- see \cite[Appendix A]{LN20}.}: 
\begin{align}
& \Gamma\subset\mathbb{R}^n\text{ is an open, convex, connected symmetric cone with vertex at 0}, \label{21} \\ 
& \Gamma_n^+ = \{\lambda\in\mathbb{R}^n: \lambda_i > 0~\forall\,1\leq i \leq n\} \subseteq \Gamma \subseteq \Gamma_1^+ =\{\lambda\in\mathbb{R}^n : \lambda_1+\dots+\lambda_n > 0\}, \label{22}  \\
& f\in C^\infty(\Gamma)\cap C^0(\overline{\Gamma}) \text{ is concave, 1-homogeneous and symmetric in the }\lambda_i, \label{23}  \\
& f>0 \text{ in }\Gamma, \quad f = 0 \text{ on }\partial\Gamma, \quad f_{\lambda_i} >0 \text{ in } \Gamma \text{ for }1 \leq i \leq n. \label{24}
\end{align}
By a classical result of Calder\'on \& Zygmund \cite{CZ61}, functions $u\in W^{2,p}_{\operatorname{loc}}(\Omega)$ for $p>n/2$ are pointwise twice differentiable a.e.~in $\Omega$, and so for such functions the equation \eqref{1} is well-defined.

The motivation behind \eqref{1} comes from conformal geometry: if $n\geq 3$ and $g_{ij}=u^{-2}\delta_{ij}$ is a metric conformal to the Euclidean metric on $\Omega$, then $A_g\defeq u^{-1}A_u$ is the $(0,2)$-Schouten tensor of $g$, which arises in the Ricci decomposition of the Riemann curvature tensor and is given by the formula
\begin{equation}\label{11}
A_g = \frac{1}{n-2}\bigg(\operatorname{Ric}_g - \frac{R_g}{2(n-1)}g\bigg).
\end{equation}
Here, $\operatorname{Ric}_g$ and $R_g$ denote the Ricci tensor and scalar curvature of the metric $g$, respectively. In the Euclidean setting, the quantity $A_u$ plays a central role in the characterisation of conformally invariant operators on $\mathbb{R}^n$ in dimensions $n\geq 3$ (see \cite{LL03}), and in the characterisation of M\"obius-invariant operators in $\mathbb{R}^2$ (see \cite{LLL21}). 

Of particular interest is \eqref{1} in the case $(f,\Gamma)=(\sigma_k^{1/k},\Gamma_k^+)$ for $1\leq k\leq n$, where
\begin{equation*}
\sigma_k(\lambda_1,\dots,\lambda_n) \defeq \sum_{1\leq {i_1}<\dots<{i_k}\leq n} \lambda_{i_1}\cdots\lambda_{i_k}
\end{equation*}
is the $k$'th elementary symmetric polynomial and 
\begin{equation*} 
\Gamma_k^+ = \{\lambda=(\lambda_1,\dots,\lambda_n)\in\mathbb{R}^n:\sigma_j(\lambda)>0~\text{for all }1\leq j\leq k\}.
\end{equation*} 

\noindent When $k=1$ and $\psi(x,z) = z^{-1}$, \eqref{1} is the Yamabe equation in the case of positive scalar curvature on Euclidean domains. For $k\geq 2$, \eqref{1} is fully nonlinear and encompasses the so-called $\sigma_k$-Yamabe equation (where $\psi(x,z)=z^{-1}$) on Euclidean domains, whose study on Riemannian manifolds was initiated by Viaclovsky in \cite{Via00a}. Note that, for general $f$ satisfying \eqref{24}, \eqref{1} is an elliptic equation, although non-uniformly elliptic \textit{a priori}. Fully nonlinear elliptic equations involving the eigenvalues of the Hessian were first considered in \cite{CNS3}.

 Important in the study of \eqref{1} are local first and second derivative estimates on solutions. Such \textit{a priori} estimates (depending on $C^0$ bounds) have been established for the $\sigma_k$-Yamabe equation and other equations of the form \eqref{1} on general Riemannian manifolds by Chen \cite{Che05}, Guan \& Wang \cite{GW03b}, Jin, Li \& Li \cite{JLL07}, Li \& Li \cite{LL03}, Li \cite{Li09} and Wang \cite{Wan06}, for example (see also the work of Viaclovsky \cite{Via02} for global estimates). On the other hand, the regularity theory for fully nonlinear Yamabe-type equations is less well-developed; for a partial list of works addressing related regularity problems, see \cite{CHY05, Gonz05, Gonz06, LN20b, LNX22, DN21, DN22}. 
 
  In joint work with Nguyen in \cite{DN21}, we studied the regularity of $W^{2,p}_{\operatorname{loc}}(\Omega)$ solutions to \eqref{1} in the case $(f,\Gamma) = (\sigma_k^{1/k},\Gamma_k^+)$, assuming $2\leq k \leq n$ and $p>kn/2$. The purpose of this paper is to both weaken the regularity assumptions in \cite{DN21} and extend the scope of the regularity theory to more general operators $f$. In addition to \eqref{23} and \eqref{24}, we introduce one more condition on $(f,\Gamma)$, which is related to the lower bound on the Sobolev exponent $p$ that we will impose on our solution $u\in W^{2,p}_{\operatorname{loc}}(\Omega)\cap C^{0,1}_{\operatorname{loc}}(\Omega)$. As we will see, this condition is satisfied by the $\sigma_k$ operators, their quotients and other important examples. 
  
  To formulate this condition, we fix $(f,\Gamma)$ satisfying \eqref{21}--\eqref{24}, and for a symmetric matrix $A$, we denote by $F(A)$ the matrix with entries 
 \begin{equation*}
 F(A)^{ij} = \frac{\partial }{\partial A_{ij}}f(\lambda(A)). 
 \end{equation*}
 Note that by \eqref{24}, $F(A)$ is positive definite if $\lambda(A)\in\Gamma$. Our condition is then as follows: there exist constants $C> 0$ and $\gamma\geq 0$ (depending only on $(f,\Gamma)$) such that
 \begin{equation}\label{26}
 \frac{[\operatorname{tr}(F(A))]^n}{\det (F(A))} \leq C \bigg(\frac{\operatorname{tr}(A)}{f(\lambda(A))}\bigg)^\gamma \quad \text{for all }A \text{ with }\lambda(A)\in\Gamma.\medskip 
 \end{equation}
 Before stating our main result, we give some examples of $(f,\Gamma)$ satisfying \eqref{26} for some $C>0$ and $\gamma\geq 0$:
 \begin{Ex}\label{50}
 	When $(f,\Gamma) = ((\sigma_k/\sigma_l)^{1/(k-l)},\Gamma_k^+)$ for some $0 \leq l < k \leq n$ and $k\geq 2$ (with the convention that $\sigma_0 =1$), \eqref{26} is satisfied with $\gamma = (k-1)\max\{k-l, 2\}$ (see \cite[Proposition 4.2]{BL} for a proof of this fact). In particular, when $(f,\Gamma) = (\sigma_k^{1/k},\Gamma_k^+)$ for $2\leq k \leq n$, \eqref{26} is satisfied with $\gamma = k(k-1)$. 
 \end{Ex}

\begin{Ex}\label{51}
	If, in addition to \eqref{21} and \eqref{22}, the cone $\Gamma$ satisfies $(1,0,\dots,0)\in\Gamma$, then \eqref{26} is satisfied with $\gamma =0$. This follows immediately from \cite[Proposition A.1]{LN20}, which asserts the existence of a constant $\nu\in(0,1)$ such that 
	\begin{equation*}
	\frac{\partial f}{\partial\lambda_i}(\lambda) \geq \nu \sum_{j=1}^n\frac{\partial f}{\partial \lambda_j}(\lambda)\quad\text{for all }i=1,\dots,n \text{ and }\lambda\in\Gamma,
	\end{equation*}
	and the fact that $\frac{\partial f}{\partial \lambda_i}$ are precisely the eigenvalues of $F$. 
\end{Ex}

 Our main result in this paper is as follows: 
 
 \begin{thm}\label{A}
 	Let $\Omega$ be a domain in $\mathbb{R}^n$ ($n\geq 2$) and let $\psi=\psi(x,z)\in C^{1,1}_{\operatorname{loc}}(\Omega\times\mathbb{R})$ be a positive function. Suppose that $(f,\Gamma)$ satisfies \eqref{21}--\eqref{24} and assume there exist constants $C>0$ and $\gamma\geq 0$ such that \eqref{26} holds. Then any positive solution $u\in W^{2,p}_{\operatorname{loc}}(\Omega)\cap C^{0,1}_{\operatorname{loc}}(\Omega)$ to \eqref{1} with
 	\begin{equation*}
 	\begin{cases}
 	p = n & \text{if }\gamma<n \\
 	p > \gamma & \text{if }\gamma \geq n
 	\end{cases}
 	\end{equation*}
 	belongs to $C^{1,1}_{\operatorname{loc}}(\Omega)$. Moreover, for any concentric balls $B_R\subset B_{3R}\Subset\Omega$ there exists a constant $C$ depending only on $n, p, R, \psi, f, \Gamma$ and an upper bound for $\|\ln u\|_{C^{0,1}(B_{3R})}+ \|\nabla^2 u\|_{L^p(B_{3R})}$ such that 
 	\begin{equation}\label{28}
 	\|\nabla^2 u \|_{L^\infty(B_{R})}\leq C. 
 	\end{equation}
 \end{thm}

\begin{rmk}
	Once the estimate \eqref{28} is established, \eqref{1} becomes uniformly elliptic, and $C_{\operatorname{loc}}^{2,\alpha}(\Omega)$ regularity follows from the concavity assumption in \eqref{23} and the regularity theory of Evans-Krylov \cite{Ev82, Kry82} (see also \cite{CC95}). Schauder estimates then yield $C_{\operatorname{loc}}^{3,\alpha}(\Omega)$ regularity for $u$, which can be bootstrapped in the usual way if one assumes additional regularity on $\psi$. 
\end{rmk}

\begin{rmk}
	In light of Morrey's embedding theorem, the $C^{0,1}_{\operatorname{loc}}(\Omega)$ assumption on $u$ in Theorem \ref{A} is superfluous when $\gamma\geq n$. 
\end{rmk}

\begin{rmk}
	We refer the reader to \cite[Appendix A]{DN21} for an example (motivated by previous work of Chang, Gursky \& Yang \cite{CGY02a}) in which the existence of a $W^{2,p}$-strong solution to a fully nonlinear Yamabe-type equation is obtained, and for which our regularity theory in Theorem \ref{A} (more precisely, Corollary \ref{44}) is applicable.
\end{rmk}

As mentioned above, in \cite[Theorem 1.1]{DN21} we proved the $C^{1,1}_{\operatorname{loc}}(\Omega)$ regularity of positive solutions $u\in W^{2,p}_{\operatorname{loc}}(\Omega)$ to \eqref{1} for $(f,\Gamma) = (\sigma_k^{1/k},\Gamma_k^+)$, assuming $2\leq k \leq n$ and $p>kn/2$. In light of this result, Example \ref{50} and Theorem \ref{A}, we therefore have:

\begin{cor}\label{44}
	Let $\Omega$ be a domain in $\mathbb{R}^n$ ($n\geq 2$), let $\psi=\psi(x,z)\in C^{1,1}_{\operatorname{loc}}(\Omega\times\mathbb{R})$ be a positive function and suppose $(f,\Gamma) = (\sigma_k^{1/k},\Gamma_k^+)$ for some $2\leq k \leq n$. Then any positive solution $u\in W^{2,p}_{\operatorname{loc}}(\Omega)\cap C^{0,1}_{\operatorname{loc}}(\Omega)$ to \eqref{1} with
	\begin{equation*}
	\begin{cases}
	p = n & \text{if } k(k-1)<n \\
	p > \min \big(k(k-1),\frac{kn}{2}\big) & \text{if }k(k-1) \geq n
	\end{cases}
	\end{equation*} 
	belongs to $C^{1,1}_{\operatorname{loc}}(\Omega)$. 
\end{cor}
\begin{rmk}
	It is easy to check that Theorem \ref{A} yields an improvement on the lower bound for $p$ assumed in \cite[Theorem 1]{DN21} when $k < \frac{n}{2}+1$. It would be interesting to determine the sharp lower bounds on $p$ for which one can obtain $C^{1,1}_{\operatorname{loc}}(\Omega)$ regularity when $(f,\Gamma) = (\sigma_k^{1/k},\Gamma_k^+)$. 
\end{rmk}

Theorem \ref{A} also encompasses examples that were not addressed in \cite{DN21} (see below for a more detailed discussion on the differences between \cite{DN21} and the present work). The first of these are the $\sigma_k$-quotient equations -- by Example \ref{50} and Theorem \ref{A}, we immediately obtain:

\begin{cor}
	Let $\Omega$ be a domain in $\mathbb{R}^n$ ($n\geq 2$), let $\psi=\psi(x,z)\in C^{1,1}_{\operatorname{loc}}(\Omega\times\mathbb{R})$ be a positive function and suppose $(f,\Gamma) = ((\sigma_k/\sigma_l)^{1/(k-l)},\Gamma_k^+)$ for some $1 \leq l < k \leq n$. Then any positive solution $u\in W^{2,p}_{\operatorname{loc}}(\Omega)\cap C^{0,1}_{\operatorname{loc}}(\Omega)$ to \eqref{1} with
	\begin{equation*}
	\begin{cases}
	p = n & \text{if } (k-1)\max\{k-l, 2\}<n \\
	p > (k-1)\max\{k-l, 2\} & \text{if }(k-1)\max\{k-l, 2\} \geq n
	\end{cases}
	\end{equation*} 
	belongs to $C^{1,1}_{\operatorname{loc}}(\Omega)$. 
\end{cor}

We also have the following consequence of Example \ref{51} and Theorem \ref{A}:

\begin{cor}\label{D}
	Let $\Omega$ be a domain in $\mathbb{R}^n$ ($n\geq 2$), let $\psi=\psi(x,z)\in C^{1,1}_{\operatorname{loc}}(\Omega\times\mathbb{R})$ be a positive function and suppose $(f,\Gamma)$ satisfies \eqref{21}--\eqref{24} and $(1,0,\dots,0)\in\Gamma$. Then any positive solution $u\in W^{2,n}_{\operatorname{loc}}(\Omega)\cap C^{0,1}_{\operatorname{loc}}(\Omega)$ to \eqref{1} belongs to $C^{1,1}_{\operatorname{loc}}(\Omega)$. 
\end{cor}

As shown in \cite[Appendix A]{DN22}, if $\Gamma$ satisfies \eqref{21} and \eqref{22}, then $(1,0,\dots,0)\in\Gamma$ if and only if $\Gamma = (\widetilde{\Gamma})^\tau$ for some $\tau\in(0,1)$ and some $\widetilde{\Gamma}$ satisfying \eqref{21} and \eqref{22}, where $(\widetilde{\Gamma})^\tau = \{\lambda\in\mathbb{R}^n:\tau\lambda+(1-\tau)\sigma_1(\lambda)e\in\widetilde{\Gamma}\}$ and $e = (1,\dots,1)$. For $-\infty<t<1$, define the trace-modified Schouten tensor \medskip
\begin{equation*}
A_g^t = \frac{1}{n-2}\bigg(\operatorname{Ric}_g - \frac{tR_g}{2(n-1)}g\bigg)\medskip
\end{equation*}
(this quantity was introduced independently by Li \& Li in \cite{LL03} and Gursky \& Viaclovsky in \cite{GV03}). Then $A_g^t = \tau^{-1}[\tau A_g + (1-\tau)\sigma_1(g^{-1}A_g)g]$ for $\tau = (1+ \frac{1-t}{n-2})^{-1}\in (0,1)$, and hence Corollary \ref{D} encompasses the so-called trace-modified $\sigma_k$-Yamabe equation on Euclidean domains; in terms of the conformal factor $u$, this corresponds to considering \eqref{1} with $\psi(x,z) =z^{-1}$, $(f,\Gamma) = (\sigma_k^{1/k},\Gamma_k^+)$ and $A_u$ replaced by 
\begin{equation*}
A_u^t = \nabla^2 u + \frac{1-t}{n-2}\Delta u\,I - \frac{2n-t n -2}{n-2}\frac{|\nabla u|^2}{2u}I.
\end{equation*}

For the remainder of the introduction, we briefly compare the methods of the present paper and those of \cite[Theorem 1.1]{DN21}. Our arguments in \cite{DN21} used the divergence structure of the $\sigma_k$-Yamabe equation on Euclidean domains, and consisted of an integrability improvement argument followed by Moser iteration. Similar methods were previously utilised by Urbas \cite{Urb00, Urb01} in proving $C^{1,1}_{\operatorname{loc}}(\Omega)$ regularity of $W^{2,p}_{\operatorname{loc}}(\Omega)$ solutions to $k$-Hessian equations. On the other hand, our proof of Theorem \ref{A} does not rely on any divergence structure, and instead involves an application of the Alexandrov-Bakelman-Pucci (ABP) estimate, as inspired by the work of Bao et.~al.~\cite{BCGJ03} on quotient Hessian equations (see also \cite{BL}). 

We point out that, although $-\frac{|\nabla u|^2}{2u}I$ is a lower order term in the definition of $A_u$, its presence leads to terms in our estimates which are formally of third order and must be dealt with carefully. In \cite{DN21}, these terms were dealt with via a delicate cancellation phenomenon, which used the aforementioned divergence structure of the $\sigma_k$-Yamabe equation on Euclidean domains. In the proof of Theorem \ref{A}, we estimate the third order terms more directly by instead appealing to properties of the concave envelope of a suitable function (involving second order difference quotients of our solution $u$) and a discrete Bochner-type formula (which we previously derived in \cite{DN21}). The details will be provided in Section \ref{3}. \newline

\noindent \textbf{Acknowledgements:} The author would like to thank Luc Nguyen for helpful discussions regarding this work. Part of this work was carried out whilst the author was supported by EPSRC grant number EP\!/L015811\!/1.

\section{Proof of Theorem \ref{A}}\label{3}

In this section we prove Theorem \ref{A}. In fact we prove a slightly more general statement. To this end, let $H\in C^{1,1}_{\operatorname{loc}}(\Omega\times \mathbb{R})$ be a real -valued function, assumed to be positive or identically zero, and let $J\in C^{1,1}_{\operatorname{loc}}(\Omega\times\mathbb{R};\operatorname{Sym}_n(\mathbb{R}))$ be a symmetric matrix-valued function. Then define
\begin{equation*}
A_{H,J}[u] = \nabla^2 u - H[u]|\nabla u|^2 I + J[u],
\end{equation*}
where $H[u](x) = H(x,u(x))\in\mathbb{R}$ and $J[u](x) = J(x,u(x))\in\operatorname{Sym}_n(\mathbb{R})$. Suppose also that $\psi_1\in C^{1,1}_{\operatorname{loc}}(\Omega\times\mathbb{R})$ is positive and $\psi_2\in C^{1,1}_{\operatorname{loc}}(\Omega\times\mathbb{R})$ is either positive or identically zero, and denote $\psi_i[u](x) = \psi_i(x,u(x))\in\mathbb{R}$ for $i=1,2$. Consider the equation 
\begin{equation}\label{25}
f\big(\lambda(A_{H,J}[u])\big) = \psi_1[u] + \psi_2[u]|\nabla u|^2>0, \quad \lambda(A_{H,J}[u])\in\Gamma \quad\text{a.e. in }\Omega. 
\end{equation}
We prove:

\begin{customthm}{\ref{A}$'$}\label{A'}
\textit{Let $\Omega$ be a domain in $\mathbb{R}^n$ ($n\geq 2$) and let $H,J,\psi_1$ and $\psi_2$ be as above. Suppose that $(f,\Gamma)$ satisfies \eqref{21}--\eqref{24} and assume there exist constants $C>0$ and $\gamma\geq 0$ such that \eqref{26} holds. Then any solution $u\in W^{2,p}_{\operatorname{loc}}(\Omega)\cap C^{0,1}_{\operatorname{loc}}(\Omega)$ to \eqref{1} with
	\begin{equation*}
	\begin{cases}
	p = n & \text{if }\gamma<n \\
	p > \gamma & \text{if }\gamma \geq n
	\end{cases}
	\end{equation*}
	belongs to $C^{1,1}_{\operatorname{loc}}(\Omega)$. Moreover, for any concentric balls $B_R\subset B_{3R}\Subset\Omega$ there exists a constant $C$ depending only on $n, p, R, \psi_1,\psi_2, H, J, f, \Gamma$ and an upper bound for $\| u\|_{C^{0,1}(B_{3R})}+ \|\nabla^2 u\|_{L^p(B_{3R})}$ such that 
	\begin{equation*}
	\|\nabla^2 u \|_{L^\infty(B_{R})}\leq C. 
	\end{equation*}}
\end{customthm}

\begin{rmk}
	Theorem \ref{A} is a special case of Theorem \ref{A'} with $H[u] = \frac{1}{2u}$, $J\equiv 0$ and $\psi_2\equiv0$. 
\end{rmk}

\begin{rmk}
	In the aforementioned works \cite{Urb00, Urb01, BCGJ03, BL}, Hessian equations of the form $(\sigma_k/\sigma_l)^{1/(k-l)}(\lambda(\nabla^2 u(x))) = \psi(x)>0$ are considered (with $l=0$ in \cite{Urb00, Urb01}). As far as the author is aware, Theorem \ref{A'} is new even when $H\equiv 0$, $J\equiv 0$ and $(f,\Gamma) = (\sigma_k^{1/k},\Gamma_k^+)$, since we allow the RHS in \eqref{25} to also depend on $u$ and $\nabla u$. 
\end{rmk}

To give one example covered by Theorem \ref{A'} but not Theorem \ref{A}, we note that equation \eqref{25} encompasses the following equation of Monge-Amp\`ere type on the round sphere $(S^n,g_0)$:
	\begin{align*}
	\operatorname{det}^{1/n}\bigg(g_0^{-1}\bigg(\nabla_{g_0}^2 u - \frac{|\nabla_{g_0} u|_{g_0}^2}{2u}g_0 + \frac{u}{2}g_0\bigg)\bigg) = \frac{|\nabla_{g_0}u|_{g_0}^2 + u^2}{2u}\phi,
	\end{align*}
	where $\phi=\phi(x)>0$ is given and one looks for a positive solution $u$. This equation has been studied previously in the context of geometric optics -- see for instance equation (1.11) in \cite{Wang96}, therein taking the distribution density $f$ to be constant.

\subsection{Notation and outline of the proof}\label{38} 
~\medskip

We now give a brief outline of the proof of Theorem \ref{A'}, which will establish notation and highlight the main steps. We start by fixing a solution $u\in W^{2,p}_{\operatorname{loc}}(\Omega)\cap C^{0,1}_{\operatorname{loc}}(\Omega)$ to \eqref{25}. For a fixed unit vector $\xi\in\mathbb{R}^n$ and small $h\in\mathbb{R}\backslash\{0\}$, we define the first order difference quotient $\nabla_\xi^h u(x) \defeq h^{-1} (u(x+h\xi) - u(x))$ and the second order difference quotient 
\begin{equation*}
\Delta_{\xi\xi}^h u \defeq \nabla_\xi^h (\nabla_\xi^{-h}u(x)) = \frac{u(x+h\xi) - 2u(x) + u(x-h\xi)}{h^2}.
\end{equation*}
To prove Theorem \ref{A'}, it suffices to obtain (for sufficiently small $h$) an upper bound for $\Delta^h_{
\xi\xi}u$ on $B_R$ which is independent of $h$. Indeed, this implies an upper bound for $\Delta u$ on $B_R$, the assumption $\lambda(A_{H,J}[u])\in\Gamma\subseteq\Gamma_1^+$ implies a lower bound for $\Delta u$, and the full Hessian bound then follows from writing $\nabla_i\nabla_j u = \frac{1}{2}(2\nabla_{\xi}\nabla_{\xi}u - \nabla_i\nabla_i u -\nabla_j\nabla_j u)$ for $\xi = \frac{1}{\sqrt{2}}(e_i+e_j)$, where $\nabla_{\xi}$ is the directional derivative in the direction $\xi$ and $\{e_i\}_{1\leq i\leq n}$ is the standard basis on $\mathbb{R}^n$. Our upper bound for $\Delta^h_{\xi\xi} u$ on $B_R$ will depend on the $L^p$ norm of $\Delta u$ on $B_{3R}$. 

To this end, define on $B_{2R}$ (which we assume to be centred at the origin) the function $v= \eta \Delta^h_{\xi\xi} u$, where 
\begin{equation}\label{e}
\eta(x) = \bigg(1-\frac{|x|^2}{4R^2}\bigg)^\beta
\end{equation}
and $\beta>2$ is a constant to be determined later. Recall that the linearised operator 
\begin{equation*}
F^{ij} \defeq \frac{\partial }{\partial A_{ij}}f(\lambda(A_{H,J}[u]))
\end{equation*}
is positive definite a.e.~in $\Omega$ by the ellipticity assumption in \eqref{24}. The first main step of our proof is to obtain an upper bound for $-F^{ij}\nabla_i\nabla_jv$ (which is formally of fourth order in the derivatives of $u$) in terms of $\operatorname{tr}(F)|\nabla(\Delta^h u)|$ (which is formally of third order), $\operatorname{tr}(F)|\Delta^h u|$ (which is formally of second order) and lower order terms. More precisely, we will prove in Section \ref{29} the following lemma:

\begin{lem}\label{t26}
	Let $\Omega, f,\Gamma, H,J,\psi_1$ and $\psi_2$ be as in the statement of Theorem \ref{A'} (with $(f,\Gamma)$ not necessarily satisfying \eqref{26}). Then for any ball $B_{2R}\Subset\Omega$ and any solution $u\in W^{2,p}_{\operatorname{loc}}(\Omega)\cap C^{0,1}_{\operatorname{loc}}(\Omega)$ to \eqref{25} with $p>n/2$, it holds that
	\begin{align}\label{e4'}
	-F^{ij}\nabla_i\nabla_j v  & \leq C\operatorname{tr}(F)\Big(\eta\big(|\nabla(\Delta^h_{\xi\xi} u)|  + |\Delta^h_{\xi\xi} u| + 1\big) + |\nabla\eta||\nabla(\Delta^h_{\xi\xi} u)| + |\nabla^2 \eta||\Delta^h_{\xi\xi} u|\Big) \nonumber \\
	& \quad + C\eta\big(|\nabla(\Delta^h_{\xi\xi} u)| + |\Delta^h_{\xi\xi} u|+1\big)
	\end{align}
	a.e.~in $B_{2R}$, where $C$ is a constant depending only on $n,R,H,J,\psi_1,\psi_2$ and an upper bound for $\|u\|_{C^{0,1}(B_{2R})}$. 
\end{lem}

Lemma \ref{t26} is the main new ingredient in the proof of Theorem \ref{A'}; once the estimate \eqref{e4'} is established, the proof of Theorem \ref{A'} then proceeds similarly to that of \cite{BCGJ03}, but with some extra terms. We summarise this argument now (the details will be given in Section \ref{27}). The main point is that on the upper contact set of $v$ in $B_{2R}$, defined by
\begin{equation}\label{e19}
\Gamma_v^+(B_{2R}) = \{x\in B_{2R}: v(z)\leq v(x) + \nu\cdot(z-x) ~\text{for all }z\in B_{2R},~\text{for some }\nu\in\mathbb{R}^n\},
\end{equation}
one can bound $|\nabla(\Delta^h_{\xi\xi} u)|$ from above in terms of $|\Delta^h_{\xi\xi} u|$. More precisely, we have:

\begin{lem}[\cite{BCGJ03}]\label{37}
	Almost everywhere on $\Gamma_v^+(B_{2R})$, it holds that
	\begin{equation}\label{35}
	\eta|\nabla(\Delta^h_{\xi\xi} u)| \leq (1+\beta)R^{-1}\eta^{-\frac{1}{\beta}}v. 
	\end{equation}
\end{lem}

After substituting \eqref{35} into \eqref{e4'}, dividing through by $(\operatorname{det}F^{ij})^{1/n}$, applying \eqref{26} and carrying out some simple calculations, we will obtain the estimate 
\begin{equation*}
0  \leq \frac{-F^{ij}\nabla_i\nabla_j v}{(\det F^{ij})^{1/n}} \leq  C\bigg( \frac{v}{R\eta^{\frac{1}{\beta}}} + v + \eta  +  \frac{v}{R^2\eta^{\frac{2}{\beta}}}\bigg)(\Delta u+C)^{\gamma/n} \quad \text{a.e. on }\Gamma_v^+(B_{2R}). 
\end{equation*}
An application of the ABP estimate and some further calculations then yields an upper bound for $v$ on $B_{2R}$, and hence an upper bound for $\Delta^h_{\xi\xi} u$ on $B_R$, as required. For later reference, we recall the ABP estimate as follows:

\begin{thm}[see e.g.~{\cite[Chapter 9]{GT}}]\label{alex}
	Suppose $a^{ij}$ is measurable and positive definite a.e.~on a smooth bounded domain $\Sigma\subset\mathbb{R}^n$. Then there exists a constant $C=C(n)$ such that for any $\phi\in W^{2,n}_{\operatorname{loc}}(\Sigma)\cap C^0(\overline\Sigma)$ with $\phi\equiv 0$ on $\partial\Sigma$, one has
	\begin{equation*}
	\sup_\Sigma \phi \leq Cd\bigg(\int_{\Gamma_v^+(\Sigma)}\frac{(-a^{ij}\nabla_i\nabla_j\phi)^n}{\operatorname{det}(a^{ij})}\,dx\bigg)^{1/n},
	\end{equation*}
	where $d$ is the diameter of $\Sigma$.
\end{thm} 

Throughout the rest of the paper, we use summation convention over indices appearing as both a subscript and superscript. We may also suppress the phrase `a.e.'~for pointwise calculations involving second derivatives. 

\subsection{Proof of Lemma \ref{t26}}\label{29}

~\medskip 

In this section we prove Lemma \ref{t26} -- as remarked above, this is the main new ingredient in the proof of Theorem \ref{A}. 

\begin{proof}[Proof of Lemma \ref{t26}]
	By concavity of $f$, we have 
	\begin{align}\label{17}
	f\big(\lambda(A_{H,J}[u](x\pm h\xi))\big) - f\big(&\lambda(A_{H,J}[u](x))\big) \nonumber \\
	&  \leq F^{ij}(x)\big(A_{H,J}[u](x\pm h\xi) - A_{H,J}[u](x)\big)_{ij}
	\end{align}
	a.e.~in $B_{2R}$. Summing the two inequalities in \eqref{17} and dividing through by $h^2$, we therefore obtain
	\begin{equation}\label{18}
	\Delta^h_{\xi\xi} f(\lambda(A_{H,J}[u]))(x) \leq F^{ij}(x)\Delta^h_{\xi\xi}\big(\nabla^2 u - H[u]|\nabla u|^2I + J[u]\big)_{ij}(x). 
	\end{equation} 
	Substituting the equation \eqref{25} into the LHS of \eqref{18}, and commuting difference quotients with derivatives on the RHS of \eqref{18}, we see that
	\begin{equation}\label{e1}
	\Delta^h_{\xi\xi} \psi[u] \leq  F^{ij}\nabla_i\nabla_j\Delta^h_{\xi\xi} u -  \operatorname{tr}(F) \Delta^h_{\xi\xi} \big(H[u]|\nabla u|^2\big) + F^{ij}\Delta^h_{\xi\xi}(J[u])_{ij},
	\end{equation}
	where $\psi[u](x) \defeq \psi_1[u](x) + \psi_2[u](x)|\nabla u(x)|^2$. It then follows that
	\begin{align}\label{e2}
	F^{ij}\nabla_i\nabla_j v  & = F^{ij}\bigg(\eta \nabla_i\nabla_j\Delta^h_{\xi\xi} u + 2\nabla_i\eta \nabla_j\Delta^h_{\xi\xi} u + (\Delta^h_{\xi\xi} u)\nabla_i\nabla_j\eta\bigg) \nonumber \\
	& \stackrel{\eqref{e1}}{\geq} \eta \operatorname{tr}(F)\Delta^h_{\xi\xi}\big(H[u]|\nabla u|^2\big) + \eta\Delta^h_{\xi\xi} \psi[u] - \eta F^{ij}\Delta^h_{\xi\xi}(J[u])_{ij} \nonumber \\
	& \qquad + 2F^{ij}\nabla_i\eta \nabla_j\Delta^h_{\xi\xi} u + \Delta^h_{\xi\xi} uF^{ij}\nabla_i\nabla_j\eta. 
	\end{align}
	To obtain the desired estimate \eqref{e4'} from \eqref{e2}, it suffices to prove the following estimates: 
	\begin{flalign}\label{e17}
	\mathrm{\textit{Estimate~1:}} &&   \Delta^h_{\xi\xi}(H[u]|\nabla u|^2) &  \geq - C|\nabla(\Delta^h_{\xi\xi} u)| - C|\Delta^h_{\xi\xi} u| - C, & 
	\end{flalign}
	\begin{flalign}\label{36}
	\mathrm{\textit{Estimate~2:}} && \Delta^h_{\xi\xi} \psi[u] & \geq -C|\nabla(\Delta^h_{\xi\xi} u)| -C|\Delta^h_{\xi\xi} u| - C, &
	\end{flalign}
	\begin{flalign}\label{39}
	\mathrm{\textit{Estimate~3:}} && \Delta^h_{\xi\xi}(J[u])_{ij} & \geq  -C|\Delta^h_{\xi\xi} u| - C. & \medskip 
	\end{flalign}

	\noindent\textit{Proof of Estimate 1:} If $H$ is identically zero then the estimate is trivial, so suppose $H>0$. For a function $w$, define $w^h_\xi$ by $w^h_\xi(x)= w(x+h\xi)$. We use the following discrete Bochner-type formula, which we previously derived in \cite[Lemma 4.16]{DN21}: 
	\begin{align}\label{e10}
	\Delta_{\xi\xi}^h\big(H[u]|\nabla u|^2\big) & = 2H[u]\nabla^i u\nabla_i\Delta_{\xi\xi}^h u + (H[u])^{-h}_\xi\big|\nabla\nabla_\xi^{-h} u\big|^2 + (H[u])^h_\xi\big|\nabla\nabla_\xi^h u\big|^2 \nonumber \\
	& \quad + \nabla_\xi^{-h}\nabla_i u \nabla^i u \nabla_\xi^{-h}H[u] + \nabla_\xi^h\nabla^i u \nabla_i u \nabla_\xi^h H[u] \nonumber \\
	& \quad + \nabla_\xi^h\Big(\nabla_i u(\nabla^i u)_\xi^{-h}\nabla_\xi^{-h}H[u]\Big). 
	\end{align} 
	By applying the product rule for difference quotients twice, we further compute the bottom line of \eqref{e10} as follows:
	\begin{align}\label{40}
	\nabla_\xi^h\Big(\nabla_i u(\nabla^i u)_\xi^{-h}\nabla_\xi^{-h}H[u]\Big) & = \nabla^i u \nabla_\xi^h H[u] \nabla_\xi^h\nabla_i u + \nabla_i u\nabla_\xi^h\Big((\nabla^i u)^{-h}_\xi\nabla_\xi^{-h}H[u]\Big) \nonumber \\
	& = \nabla^i u \nabla_\xi^h H[u] \nabla_\xi^h\nabla_i u + |\nabla u|^2\Delta_{\xi\xi}^h H[u] \nonumber \\
	& \qquad + \nabla_i u\nabla_\xi^{-h}H[u]\nabla_\xi^{-h}\nabla^i u.
	\end{align}
	It follows after substituting \eqref{40} into \eqref{e10} that
	\begin{align}\label{e13}
	\Delta_{\xi\xi}^h(H[u]|\nabla u|^2) & \geq 2H[u]\nabla^i u\nabla_i\Delta_{\xi\xi}^h u + (H[u])^{-h}_\xi\big|\nabla\nabla_\xi^{-h} u\big|^2 + (H[u])^h_\xi\big|\nabla\nabla_\xi^h u\big|^2 \nonumber \\
	& \quad - 2|\nabla\nabla_\xi^{-h} u||\nabla u||\nabla_\xi^{-h}H[u]| -  2|\nabla \nabla_\xi^h u||\nabla u||\nabla_\xi^h H[u]|  \nonumber \\
	& \quad + |\nabla u|^2\Delta_{\xi\xi}^h H[u]. 
	\end{align}
	Now, since $H=H(x,z)>0$ and $u$ is continuous, $H[u]$ is bounded uniformly away from zero on any compact subset of $\Omega$, and in particular $(H[u])^{\pm h}_\xi \geq C_0>0$ in $B_{2R}$ for some constant $C_0$ independent of $h$. Moreover, since we assume $\nabla u\in L^\infty_{\operatorname{loc}}(\Omega)$, there exists a constant $C_1$ independent of $h$ such that 
	\begin{align}\label{41}
	- 2|\nabla\nabla_\xi^{-h} u||\nabla u||\nabla_\xi^{-h}H[u]| -&  2|\nabla \nabla_\xi^h u||\nabla u||\nabla_\xi^h H[u]| \nonumber \\
	& \qquad \geq -\frac{C_0}{2}|\nabla \nabla_\xi^{-h} u|^2 - \frac{C_0}{2}|\nabla\nabla_\xi^{h} u|^2 - C_1
	\end{align} 
	a.e.~in $B_{2R}$. Substituting \eqref{41} into \eqref{e13} and dropping the positive terms involving $|\nabla \nabla_\xi^{\pm h}u|^2$, we obtain the estimate
	\begin{align}\label{e16}
	\Delta^h_{\xi\xi}(H[u]|\nabla u|^2) & \geq 2H[u]\nabla^i u\nabla_i\Delta^h_{\xi\xi} u + |\nabla u|^2\Delta^h_{\xi\xi} H[u] - C \nonumber \\
	& \geq -C|\nabla(\Delta^h_{\xi\xi} u)| + |\nabla u|^2\Delta^h_{\xi\xi} H[u] - C .
	\end{align}
	
	To obtain \eqref{e17} from \eqref{e16}, it remains to show that $\Delta^h_{\xi\xi} H[u] \geq -C|\Delta^h_{\xi\xi} u| - C$. The argument is similar to that given in our proof of \cite[Lemma 4.10]{DN21}. First, by the $C^{1,1}$ regularity of $H$ on $B_{2R}$ (which implies that $H=H(x,z)$ is semi-convex in the $z$-variable), we can assert the existence of a constant $C_2$ such that
	\begin{equation*}
	H(x,u_1) \geq  H(x,u_2) + \frac{\partial H}{\partial z}(x,u_2)(u_1-u_2) - C_2|u_1-u_2|^2
	\end{equation*}
	for all $(x,u_i)\in \overline{B}_{2R}\times[-M,M]$, where $M$ is an upper bound for $\|u\|_{C^{0,1}(B_{2R})}$. Denoting $x^\pm = x\pm h\xi$, we therefore have 
	\begin{align}\label{32}
	H(x^+, u(x^+)) - H(x^+, u(x)) & \geq \frac{\partial H}{\partial z}(x^+,u(x))\big(u(x^+)-u(x)\big) - C_2|u(x^+)-u(x)|^2 \nonumber \\
	& \geq \frac{\partial H}{\partial z}(x,u(x))\big(u(x^+)-u(x)\big) - C_2|u(x^+)-u(x)|^2 \nonumber \\
	& \quad  - Ch|u(x^+)-u(x)| \nonumber \\
	&  \geq \frac{\partial H}{\partial z}(x,u(x))\big(u(x^+)-u(x)\big) - Ch^2,
	\end{align}
	where we have used
	\begin{equation*}
	\bigg|\frac{\partial H}{\partial z}(x,u(x)) - \frac{\partial H}{\partial z}(x^+,u(x))\bigg| \leq \|H\|_{C^{1,1}}|x^+-x| = h\|H\|_{C^{1,1}} = Ch
	\end{equation*}
	to obtain the second inequality in \eqref{32}, and the fact that $u\in C^{0,1}_{\operatorname{loc}}(\Omega)$ to obtain the last inequality in \eqref{32}. Similarly, 
	\begin{align}\label{33}
	H(x^-,u(x^-)) - H(x^-,u(x)) \geq \frac{\partial H}{\partial z}(x,u(x))\big(u(x^-)-u(x)\big) - Ch^2, 
	\end{align}
	and combining \eqref{32} and \eqref{33} we therefore obtain
	\begin{align}\label{e14}
	&\Delta_{\xi\xi}^h H[u](x) \geq \frac{\partial H}{\partial z}(x,u(x))\Delta_{\xi\xi}^h u(x) + \frac{H(x^+,u(x)) - 2H(x,u(x)) + H(x^-,u(x))}{h^2} - C. 
	\end{align}
	Now, by Lipschitz regularity of $H$, for all $z\in[-M,M]$ we also have 
	\begin{align}\label{42}
	H(x^+,z) - 2H(x,&z) + H(x^-,z)  \nonumber \\
	& = h\int_0^1\bigg( \nabla_\xi H(x+th\xi,z) -\nabla_\xi H(x-th\xi,z)\bigg)\,dt  \leq Ch^2,
	\end{align}
	where $\nabla_\xi H$ is the directional derivative of $H$ in the direction $\xi$. Taking $z=u(x)$ in \eqref{42} and substituting this back into \eqref{e14}, we obtain the estimate $\Delta^h_{\xi\xi} H[u] \geq -C|\Delta^h_{\xi\xi} u| - C$, and thus \eqref{e17} is established. \newline
	
	\noindent\textit{Proof of Estimate 2:} The proof that $\Delta^h_{\xi\xi} (\psi_2[u]|\nabla u|^2) \geq - C|\nabla(\Delta^h_{\xi\xi} u)| - C|\Delta^h_{\xi\xi} u| - C$ is exactly the same as the proof of Estimate 1. The proof that $\Delta^h_{\xi\xi} \psi_1[u] \geq - C|\Delta^h_{\xi\xi} u| - C$ is exactly the same as the estimate given for $\Delta^h_{\xi\xi} H[u]$ above. \newline
	
		\noindent\textit{Proof of Estimate 3:} The proof that $\Delta^h_{\xi\xi}(J[u])_{ij}\geq -C|\Delta^h_{\xi\xi} u| - C$ is exactly the same as the estimate given for $\Delta^h_{\xi\xi} H[u]$ above. \medskip 
		
		With these estimates established, the proof of Lemma \ref{t26} is therefore complete. 
	\end{proof}

\subsection{Proof of Theorem \ref{A'}}\label{27}
~\medskip 

In this section we complete the proof of Theorem \ref{A'} (and hence Theorem \ref{A}). In our estimates, we will repeatedly use the following bounds on the derivatives of $\eta$, defined in \eqref{e}:
\begin{equation}\label{e3}
|\nabla\eta| \leq \frac{C}{R}\eta^{1-\frac{1}{\beta}}\quad\text{and}\quad |\nabla^2\eta|\leq\frac{C}{R^2}\eta^{1-\frac{2}{\beta}},
\end{equation}
where $C=C(n,\beta)$. 

As outlined in Section \ref{38}, the two key lemmas in the proof of Theorem \ref{A'} are Lemma \ref{t26} (proved in the previous section) and Lemma \ref{37}, which was established in \cite{BCGJ03}. For the convenience of the reader, we give the proof of Lemma \ref{37} here.

\begin{proof}[Proof of Lemma \ref{37}]
	First observe that since $u\in C^{0,1}(\Omega)$, $\nabla v(x)$ exists for a.e.~$x\in B_{2R}$. For such $x\in\Gamma_v^+(B_{2R})$ satisfying $\nabla v(x) \not=0$, let $z\in \partial B_{2R}$ be such that
	\begin{equation*}
	\frac{z-x}{|z-x|}  = - \frac{\nabla v(x)}{|\nabla v(x)|}. 
	\end{equation*}
	Since $v=0$ on $\partial B_{2R}$ and $|z-x| \geq |z| -|x| = 2R - |x| \geq R\eta^\frac{1}{\beta}$ (the last inequality following from the definition of $\eta$), we thus have for such $x\in \Gamma_v^+(B_{2R})$ that 
	\begin{equation}\label{t49}
	v(x) \geq v(z) - \nabla v(x)\cdot(z-x)  = -\nabla v(x)\cdot(z-x) = |z-x||\nabla v(x)| \geq R\eta^{\frac{1}{\beta}}|\nabla v(x)|.
	\end{equation}
	It follows that at such points $x\in \Gamma_v^+(B_{2R})$, we have the estimate
	\begin{align}\label{e5}
	\eta|\nabla(\Delta^h_{\xi\xi} u)|  = |\nabla v - (\Delta^h_{\xi\xi} u)\nabla \eta| & \leq |\nabla v| + \Delta^h_{\xi\xi} u|\nabla \eta| \nonumber \\
	& \leftstackrel{\eqref{e3},\eqref{t49}}{\leq} \frac{v}{R\eta^\frac{1}{\beta}} + \frac{v}{\eta} \frac{\beta}{R}\eta^{1-\frac{1}{\beta}}  = \frac{(1+\beta)v}{R\eta^{\frac{1}{\beta}}}.
	\end{align}
Finally, for points $x\in\Gamma_v^+(B_{2R})$ such that $\nabla v(x)$ exists but is zero, it is clear that \eqref{e5} still holds. 
\end{proof}

\begin{proof}[Proof of Theorem \ref{A'}]
	Our starting point is the estimate \eqref{e4'} obtained in Lemma \ref{t26}, which implies (by \eqref{e3})
	\begin{align}\label{e4}
	-F^{ij}\nabla_i\nabla_jv & \leq C\operatorname{tr}(F)\bigg(\eta\big(|\nabla(\Delta^h_{\xi\xi} u)| +|\Delta^h_{\xi\xi} u| + 1\big) + \frac{\eta^{1-\frac{1}{\beta}}}{R}|\nabla(\Delta^h_{\xi\xi} u)| + \frac{\eta^{1-\frac{2}{\beta}}}{R^2}|\Delta^h_{\xi\xi} u|\bigg) \nonumber \\
	& \qquad + C\eta\big(|\nabla(\Delta^h_{\xi\xi} u)| +|\Delta^h_{\xi\xi} u| + 1\big). 
	\end{align}

	Next, we substitute the estimate \eqref{35} of Lemma \ref{37} into the RHS of \eqref{e4}. Using the fact that $\nabla_i\nabla_j v$ is negative semi-definite a.e.~in $\Gamma_v^+(B_{2R})$, and that $F^{ij}$ is positive definite, we obtain a.e. in $\Gamma_v^+(B_{2R})$ the estimate
	\begin{align}\label{e6}
	0 & \leq -F^{ij}\nabla_i\nabla_j v \nonumber \\
	& \stackrel{\eqref{35}}{\leq} C\operatorname{tr}(F)\bigg(\frac{(1+\beta)v}{R\eta^{\frac{1}{\beta}}} + v + \eta  + \frac{(1+\beta)v}{R^2\eta^\frac{2}{\beta}} + \frac{v}{R^2\eta^\frac{2}{\beta}} \bigg) + C\bigg(\frac{(1+\beta)v}{R\eta^{\frac{1}{\beta}}} + v + \eta \bigg) \nonumber \\
	& \,\,\leq  C\operatorname{tr}(F)\bigg(\frac{v}{R\eta^{\frac{1}{\beta}}} +  v +  \eta +   \frac{v}{R^2\eta^{\frac{2}{\beta}}}\bigg) + C\bigg(\frac{v}{R\eta^{\frac{1}{\beta}}} +  v +  \eta\bigg).
	\end{align}
	It then follows from \eqref{e6}, \eqref{26} and the fact that 
	\begin{equation*}
	\operatorname{tr}(A_{H,J}[u]) = \operatorname{tr}\big(\nabla^2 u - H[u]|\nabla u|^2I + J[u]\big) \leq \Delta u + C
	\end{equation*}
	that a.e.~in $\Gamma_v^+(B_{2R})$, we have
	\begin{align}\label{e9}
	0   \leq \frac{-F^{ij}\nabla_i\nabla_j v}{(\det F^{ij})^{1/n}} & \leq \frac{C\operatorname{tr}(F)}{(\operatorname{det}F^{ij})^{1/n}}\bigg(\frac{v}{R\eta^{\frac{1}{\beta}}} +  v +  \eta +   \frac{v}{R^2\eta^{\frac{2}{\beta}}}\bigg) + \frac{C}{(\operatorname{det}F^{ij})^{1/n}}\bigg(\frac{v}{R\eta^{\frac{1}{\beta}}} +  v +  \eta\bigg)\nonumber \\
	&  \leq C\bigg(\frac{\operatorname{tr}(A_{H,J}[u])}{f(\lambda(A_{H,J}[u]))}\bigg)^{\gamma/n}\bigg( \frac{v}{R\eta^{\frac{1}{\beta}}} + v + \eta  +  \frac{v}{R^2\eta^{\frac{2}{\beta}}}\bigg) \nonumber \\
	& \quad + C\bigg(\frac{\operatorname{tr}(A_{H,J}[u])}{f(\lambda(A_{H,J}[u]))}\bigg)^{\gamma/n}\frac{1}{\operatorname{tr}(F)}\bigg(\frac{v}{R\eta^{\frac{1}{\beta}}} +  v +  \eta\bigg) \nonumber \\
	& \leq C\bigg( \frac{v}{R\eta^{\frac{1}{\beta}}} + v + \eta  +  \frac{v}{R^2\eta^{\frac{2}{\beta}}}\bigg)(\Delta u + C)^{\gamma/n},
	\end{align}
	where to reach the last line we have used the equation \eqref{25} and the fact that $\psi[u]$ and $\operatorname{tr}(F(A))= \sum_i f_{\lambda_i}(\lambda(A)) = f(\lambda(A)) +\sum_i f_{\lambda_i}(\lambda(A))(1-\lambda_i) \geq f(1,\dots,1)$ are both bounded away from zero. 
	
	With \eqref{e9} established, we are now in a position to apply the ABP estimate as stated in Theorem \ref{alex}, which yields
	\begin{align}\label{2-}
	\sup_{B_{2R}} v & \leq CR\bigg(\int_{\Gamma_v^+(B_{2R})}\frac{(-F^{ij}\nabla_i\nabla_j v)^n}{\operatorname{det}(F^{ij})}\,dx\bigg)^{1/n} \nonumber \\
	& \leq C\bigg(\int_{\Gamma_v^+(B_{2R})}(\eta^{-\frac{1}{\beta}}v)^n(\Delta u + C)^{\gamma}\bigg)^{1/n} +CR\bigg(\int_{\Gamma_v^+(B_{2R})} v^n(\Delta u + C)^{\gamma}\bigg)^{1/n} \nonumber \\
	& \quad  +CR\bigg(\int_{\Gamma_v^+(B_{2R})} \eta^n(\Delta u + C)^{\gamma}\bigg)^{1/n}+ CR^{-1}\bigg(\int_{\Gamma_v^+(B_{2R})}(\eta^{-\frac{2}{\beta}}v)^n(\Delta u + C)^{\gamma}\bigg)^{1/n}.
	\end{align}
	We estimate each of the four integrals on the RHS of \eqref{2-} in turn, starting with the last one. Writing $\eta^{-\frac{2}{\beta}}v = v^{1-\frac{2}{\beta}}(\Delta^h_{\xi\xi} u)^{\frac{2}{\beta}}$, and noting that $1-\frac{2}{\beta}>0$ (since $\beta>2$), we see
	\begin{align}\label{e21}
	\bigg(\int_{\Gamma_v^+(B_{2R})}\!(\eta^{-\frac{2}{\beta}}v)^n(\Delta u+C)^{\gamma}\!\bigg)^{1/n} \!\leq (\sup_{B_{2R}}v)^{1-\frac{2}{\beta}}\bigg(\int_{B_{2R}}|\Delta^h_{\xi\xi} u|^{\frac{2n}{\beta}}(\Delta u+C)^{\gamma}\!\bigg)^{1/n}\!\!,
	\end{align}
	where we have assumed $\sup_{B_{2R}} v \geq 0$ (otherwise we are done). We estimate the remaining three terms on the RHS of \eqref{2-} so as to put them on equal footing with \eqref{e21}. For the first term, note $\eta^{-\frac{1}{\beta}}v = v^{1-\frac{2}{\beta}}(\sqrt{\eta}|\Delta^h_{\xi\xi} u|)^{\frac{2}{\beta}} \leq v^{1-\frac{2}{\beta}}|\Delta^h_{\xi\xi} u|^{\frac{2}{\beta}}$, so
	\begin{align*}
	\bigg(\int_{\Gamma_v^+(B_{2R})}(\eta^{-\frac{1}{\beta}}v)^n(\Delta u+C)^{\gamma}\bigg)^{1/n} \leq (\sup_{B_{2R}}v)^{1-\frac{2}{\beta}}\bigg(\int_{B_{2R}}|\Delta^h_{\xi\xi} u|^{\frac{2n}{\beta}}(\Delta u+C)^{\gamma}\bigg)^{1/n}.
	\end{align*}
	Similarly, for the second term, $v=v^{1-\frac{2}{\beta}}(\eta|\Delta^h_{\xi\xi} u|)^{\frac{2}{\beta}} \leq v^{1-\frac{2}{\beta}}|\Delta^h_{\xi\xi} u|^{\frac{2}{\beta}}$, so 
	\begin{align*}
	\bigg(\int_{\Gamma_v^+(B_{2R})} v^n(\Delta u+C)^{\gamma}\bigg)^{1/n}  \leq (\sup_{B_{2R}}v)^{1-\frac{2}{\beta}}\bigg(\int_{B_{2R}}|\Delta^h_{\xi\xi} u|^{\frac{2n}{\beta}}(\Delta u+C)^{\gamma}\bigg)^{1/n},
	\end{align*}
	and for the third term on the RHS of \eqref{2-} it is easy to see that
	\begin{equation*}
	\bigg(\int_{\Gamma_v^+(B_{2R})} \eta^n(\Delta u+C)^{\gamma}\bigg)^{1/n} \leq \bigg(\int_{B_{2R}} (\Delta u+C)^{\gamma}\bigg)^{1/n} \leq C. 
	\end{equation*}
	
	Substituting the previous four estimates into \eqref{2-} and using the fact that $R$ is bounded, we therefore obtain
	\begin{align}\label{30}
	\sup_{B_{2R}} v \leq CR^{-1}(\sup_{B_{2R}}v)^{1-\frac{2}{\beta}}\bigg(\int_{B_{2R}}|\Delta^h_{\xi\xi} u|^{\frac{2n}{\beta}}(\Delta u+C)^{\gamma}\bigg)^{1/n} + CR^{-1}.
	\end{align}

	We now choose $\beta = \frac{2n}{p-\gamma}$, where $p$ is as in the statement of Theorem \ref{A'}, so that $\frac{2n}{\beta} = p-\gamma$. After applying H\"older's inequality to the integral on the RHS of \eqref{30}, dividing through by $(\sup_{B_{2R}}v)^{1-\frac{2}{\beta}}$, and using the inequality $\|\Delta_{\xi\xi}^h u\|_{L^s(B_{2R})} \leq C(n)\|\Delta u\|_{L^s(B_{3R})}$ for any $s\geq 1$ (see e.g.~\cite[Lemma 7.23]{GT}), we obtain the estimate
	\begin{align}\label{43}
	(\sup_{B_{2R}} v)^{2/\beta}\leq CR^{-1}\|\Delta u+C \|_{L^p(B_{3R})}^{p/n} + \frac{CR^{-1}}{(\sup_{B_{2R}} v)^{1-\frac{2}{\beta}}}.
	\end{align}
	If $\sup_{B_{2R}} v\leq 1$ then we are done. Supposing otherwise, \eqref{43} then implies
	\begin{align*}
	(\sup_{B_{2R}} v)^{2/\beta}\leq CR^{-1}\big(1+\|\Delta u +C\|_{L^p(B_{3R})}^{p/n} \big),
	\end{align*}
	and we therefore arrive at the estimate 
	\begin{equation}\label{100}
	\sup_{B_R} \Delta^h_{\xi\xi} u \leq CR^{-\beta/2}\big(1+\|\Delta u + C \|_{L^p(B_{3R})}^{p/n} \big)^{\beta/2}.
	\end{equation}
	As explained at the start of Section \ref{38}, the full Hessian bound for $u$ follows from \eqref{100}, and this completes the proof of Theorem \ref{A'}.
\end{proof}

\bibliography{references}{}
\bibliographystyle{siam}

\end{document}